\documentclass[reqno]{amsart}
\usepackage{hyperref}
\usepackage{graphicx}
\usepackage{curves}
\usepackage{float}
\usepackage{comment}
\usepackage{caption}
\usepackage{subcaption}
\usepackage{dirtytalk}
\usepackage{amsthm}
\usepackage{mathrsfs}
\usepackage{amsbsy}
\usepackage{mathtools}
\usepackage{cool}
\usepackage{stmaryrd}
\usepackage{relsize}
\theoremstyle{plain}
\usepackage{tikz-cd}
\usetikzlibrary{arrows,shapes}
\usetikzlibrary{decorations.text}
\usepackage{pgfplots}
\pgfplotsset{compat=1.14}
\usepackage[percent]{overpic}

\let\oldtocsection=\tocsection

\let\oldtocsubsection=\tocsubsection

\renewcommand{\tocsection}[2]{\hspace{0em}\oldtocsection{#1}{#2}}
\renewcommand{\tocsubsection}[2]{\hspace{1.75em}\oldtocsubsection{#1}{#2}}

\newcommand*{\mailto}[1]{\href{mailto:#1}{\nolinkurl{#1}}}

\newcommand\restr[2]{{
  \left.\kern-\nulldelimiterspace 
  #1 
  \vphantom{\big|} 
  \right|_{#2} 
  }}

\newtheorem{theorem}{Theorem}[section]
\newtheorem{proposition}[theorem]{Proposition}
\newtheorem{definition}[theorem]{Definition}
\newtheorem{lemma}[theorem]{Lemma}
\newtheorem{corollary}[theorem]{Corollary}
\newtheorem{remark}[theorem]{Remark}
\newtheorem{assumption}[theorem]{Assumption}

\newtheorem{example}[theorem]{Example}

\newcommand{\R}{{\mathbb R}}

\newcommand{\C}{{\mathbb C}}

\newcommand{\nn}{\nonumber}
\newcommand{\be}{\begin{equation}}
\newcommand{\ee}{\end{equation}}
\newcommand{\bea}{\begin{eqnarray}}
\newcommand{\eea}{\end{eqnarray}}
\newcommand{\btheo}{\begin{theorem}}
\newcommand{\etheo}{\end{theorem}}

\newcommand{\re}{\Re}
\newcommand{\im}{\Im}
\newcommand{\asympt}{\mathcal{O}}

\newcommand{\W}{\mathcal{W}}

\newcommand{\fisf}{\mathsf{\Phi}}

\newcommand{\ord}{{\mathcal O}}
\newcommand{\CO}{{\mathcal O}}

\newcommand{\eq}{\eqref}

\newcommand{\e}{\epsilon}

\newcommand{\lam}{\lambda}

\newcommand{\m}{\mu}
\newcommand{\al}{\alpha}

\numberwithin{equation}{section}

\title[WKB for  Dirac]
{Semiclassical WKB Problem for the non-self-adjoint Dirac operator} 

\author[S.Fujii\'e]{Setsuro Fujii\'e$^{\ast}$}
\address{$^{\ast}$ Department of Mathematical Sciences, Ritsumeikan University, Japan}
\email{fujiie@fc.ritsumei.ac.jp}

\author[N. Hatzizisis]{Nicholas Hatzizisis$^{\dag}$}
\address{$^{\dag}$ Department of Pure \& Applied Mathematics, 
University of Crete, Greece}
\email{\mailto{nhatzitz@gmail.com}}

\author[S. Kamvissis]{Spyridon Kamvissis$^{\ddag}$}
\address{$^{\ddag}$ Department of Pure \& Applied Mathematics, 
University of Crete, and Institute
of Applied and Computational Mathematics, FORTH, GR--711 10
Voutes Campus, Greece}
\email{\mailto{spyros@tem.uoc.gr}}

\begin{document}

\abstract
We review some recent rigorous results on the semiclassical behavior ($\epsilon\downarrow0$) 
of the scattering data of a 
non-self-adjoint Dirac operator with potential $A\exp\{iS/\epsilon\}$ where 
both $A$ and $S$ are differentiable  functions tending to constants as $x \to \pm \infty$.
We have either employed the so-called exact WKB method, 
or the  older WKB theory of Olver.
Our analysis is motivated
by the need to understand  the semiclassical behaviour of the focusing cubic NLS equation 
with initial data $A\exp\{iS/\epsilon\}$,  in view of the well-known fact discovered 
by Zakharov and Shabat that the spectral analysis of the Dirac operator enables us to obtain  
the solution of the NLS equation via inverse scattering theory.
\endabstract

\maketitle

\tableofcontents

\begin{center}
\textit{To Percy Deift for his $80^{th}$ birthday, with respect and gratitude!}
\end{center}

\bigskip
\bigskip

\section{Introduction}
\label{intro}

Consider the \textit{semiclassical limit} ($\epsilon\downarrow0$) of the solution to
the initial value problem of the one-dimensional \textit{nonlinear Schr\"odinger equation}
with cubic nonlinearity: 
\begin{equation}
\label{ivp-nls}
\left\{
\begin{array}{l}
i\epsilon\partial_t\psi+
\frac{\epsilon^2}{2}\partial_x^2\psi+
|\psi|^2\psi=
0, \quad(x,t)\in\R\times\R\\
\psi(x,0)=A(x)\exp\{iS(x)/\epsilon\},\quad x\in\R
\end{array}
\right.
\end{equation}
where $A$ and $S$ are real valued  functions defined on the real line. According to the
seminal discovery in \cite{zs},  this initial value problem can be studied via the
so-called \textit{inverse scattering method} for any fixed $\epsilon >0$.
The limit $\epsilon\downarrow0$ requires a careful analysis at two levels.
A necessary first  component of the semiclassical analysis is the $direct$ spectral
and scattering analysis of the associated \textit{Dirac} 
(or Zakharov-Shabat) operator $\mathfrak{D}_\epsilon$ given by
\begin{equation*}
\mathfrak{D}_\epsilon=
\begin{bmatrix}
-\frac{\epsilon}{i}\frac{d}{dx} & -iA(x)\exp\{iS(x)/\epsilon\}\\
-iA(x)\exp\{-iS(x)/\epsilon\} & \frac{\epsilon}{i}\frac{d}{dx}
\end{bmatrix}.
\end{equation*}
The second component is of course the analysis of the inverse problem, which  paradoxically
was the first to be studied rigorously. This was the object of a previous review dedicated
to Percy Deift 20 years ago (\cite{k}, \cite{kmm}, \cite{kr}, \cite{k2}). We have belatedly
returned to the direct problem!

Our main goal in this review  is to present some recent rigorous results on the semiclassical 
behavior of the \textit{scattering data} (reflection coefficient,  eigenvalues and their 
associated norming constants) for the operator $\mathfrak{D}_\epsilon$.

\section{The case of zero initial phase}
\label{S-zero}

We begin with the case where the initial phase $S(x)=0$. Our first result assumes the
existence of an analytic extension of the initial data $A(x)$. We have studied the
semiclassical asymptotics of the reflection coefficient and the eigenvalue distribution
of the Dirac operator
$$
L=
\begin{bmatrix}
-\frac \epsilon i\frac{d}{dx} & -iA(x) \\[8pt] 
-iA(x) & \frac \epsilon i\frac{d}{dx}
\end{bmatrix}
.
$$
Here, $\epsilon>0$ is the semiclassical small parameter and $A(x)$ is a function
satisfying the following.

\begin{description}
\item[(A1)]
$A(x)$ is a real positive smooth function on $\R$, and extends analytically to the
complex domain
$$
D_0=D(\rho_0,  \theta_0):=\{x\in \C; |\im \,x|<\max (\rho_0, (\tan\theta_0 )|\re \,x|)\}
$$
for some positive $\rho_0$ and $\theta_0$.
Moreover, there exists a positive $\tau$ such that, as $x\to\infty$ in $D_0$,
$$
A(x)=\ord (|x|^{-1-\tau}).
$$
\end{description}

Under this assumption, it is known that the spectrum of the non-self-adjoint operator
$L$ consists of the continuous spectrum  $\R$ and a finite number of eigenvalues coming
in complex conjugate pairs close to $i[-A_0,A_0]$, where $A_0:=\max_{x\in \R}A(x)>0$
when $\epsilon$ is small.

Before describing our results let's remind the reader of the definition of the
spectral/scattering data. For $\lambda\in \R\setminus \{0\}$, there exist a pair
of solutions ${\bf f}^r_+(x,\epsilon)$, ${\bf f}^r_-(x,\epsilon)$ to the equation
$L {\bf f}= \lambda {\bf f}$ which behave, as the real part of
$x, ~\re \,x\to \infty$ in $D_0$, like
$$
{\bf f}^r_+\sim
\begin{bmatrix}
0 \\
e^{i\lambda x/\epsilon}
\end{bmatrix}
, \quad
{\bf f}^r_-\sim
\begin{bmatrix}
e^{-i\lambda x/\epsilon} \\
0
\end{bmatrix}
,
$$
as well as a pair of solutions ${\bf f}^l_+(x,\epsilon)$, ${\bf f}^l_-(x,\epsilon)$
which behave, as $\re x\to -\infty$ in $D_0$, like
$$
{\bf f}^l_+\sim
\begin{bmatrix}
0 \\
e^{i\lambda x/\epsilon}
\end{bmatrix}
, \quad
{\bf f}^l_-\sim
\begin{bmatrix}
e^{-i\lambda x/\epsilon} \\
0
\end{bmatrix}
.
$$
These solutions are called  {\it Jost solutions}.

Each of these pairs is uniquely determined and forms  a basis of solutions.
Let $T(\lambda,\epsilon)$ be the 2$\times 2$ constant matrix depending on $\lambda$
and $\epsilon$ expressing the change of basis of these two pairs:
\begin{equation}
({\bf f}^l_+,{\bf f}^l_-)=({\bf f}^r_+,{\bf f}^r_-)T.
\end{equation}
Then $T$ is of the form
\begin{equation}
T(\lambda,\epsilon)=
\begin{bmatrix}
a(\lambda,\epsilon) & b^*(\lambda,\epsilon) \\
b(\lambda,\epsilon) & a^*(\lambda,\epsilon)
\end{bmatrix}
\end{equation}
where $a^*$, $b^*$ denote the complex conjugates of $a,b$. The reflection coefficient
$R(\lambda,\epsilon)$ is by definition
\begin{equation}
R(\lambda,\epsilon)=\frac{b(\lambda,\epsilon)}{a(\lambda,\epsilon)}.
\end{equation}
It is easy to see  that it can be expressed by wronskians of Jost solutions:
\begin{equation}
R(\lambda,\epsilon)=\frac{\W({\bf f}^r_+,{\bf f}^l_+)}{\W({\bf f}^l_+,{\bf f}^r_-)}.
\end{equation}

We have the following results \cite{fujii+kamvi}. For the reflection coefficient,
we first have the following result for $\lambda\in\R$ satisfying $|\lambda|\ge \delta$
with $\epsilon$-independent positive $\delta$.

\begin{theorem}
\label{ref1}
Assume (A1). Then, for any $\delta>0$, there exists $\sigma>0$ independent of $\epsilon$
such that
$$
|R(\lambda,\epsilon)|=\ord (e^{-\sigma/\epsilon}),
$$
as $\epsilon\to 0$ uniformly for $\lambda\in (-\infty, \delta]\cup [\delta,\infty)$.
\end{theorem}

For an estimate of the eigenvalues, we assume moreover (for simplicity) that $A(x)$ is a 
``bell-shaped" function:

\begin{description}
\item[(A2)]
$A(x)=A(-x)$ and $A'(x)<0$ for $x>0$.
\item[(A3)]
$A''(0)<0$.
\end{description}

Let $\lambda=i\mu$ with $0<\mu<A_0=A(0)$. The assumption (A2) implies that there exists
a unique positive $x^*(\mu)$ such that there are exactly two real numbers $x^*(\mu)$
and $-x^*(\mu)$  which satisfy $A(x)=\mu$. Let us define an ``action integral"
\begin{equation}
\label{action}
H(\mu)=\int_{-x^*(\mu)}^{x^*(\mu)} \sqrt{A(x)^{2}-\mu^2}dx.
\end{equation}
Then the so-called Bohr-Sommerfeld quantization rule is written as
$$
-e^{2iH(\mu)/\epsilon}=1.
$$
The meaning of this rule is that as  $\mu$ runs over  the roots of this equation
then the values $\lambda=i\mu$  will be good estimates of the eigenvalues on
$i[-A(0),A(0)]$ as  $\epsilon\to 0$.
 
To make this more precise we will show that
there exists a function $m(\mu,\epsilon)$ asymptotic to $-1$ in the semiclassical
limit $\epsilon\to 0$  such that the eigenvalues are exactly given by $\lambda=i\mu$
with the roots of $m(\mu,\epsilon)e^{2iH(\mu)/\epsilon}=1.$

\begin{theorem}
\label{ev1}
Assume (A1), (A2) and (A3).
Then there exists a function $m(\mu,\e)$ with asymptotic behavior
$$
m(\mu,\epsilon)=-1+\ord (\epsilon)
$$
as $\epsilon\to 0$ uniformly in any closed interval $I\subset (0,A(0)]$
such that $\lambda=i\mu$ where $\mu\in I$ is an eigenvalue of $L$ if and only if
\begin{equation}
m(\mu,\epsilon)e^{2iH(\mu)/\epsilon}=1.
\end{equation}
\end{theorem}

Now, it is particularly important for the detailed rigorous inverse scattering
analysis to have  delicate  estimates $near~0$. We will consider  the asymptotic
behavior of the functions $R(\lambda,\e)$ and $m(\mu,\e)$ when $\lambda>0$ or
$\mu>0$ tends to $0$  as $\e \to 0$. In such a case, we need a more precise
assumption on the asymptotic behavior of the potential $A(x)$ as $|x|\to \infty$
in $D_0$.

We define a function
\begin{equation}
\label{zx}
z(x)=i\int_0^x\sqrt{A(t)^2+\lambda^2}dt,
\end{equation}
where we take the branch of the square root such that it is positive at $t=0$.
This function is well-defined and  holomorphic at least near the origin $x=0$.
It can be extended to the set $D_0$  except at the turning points, that is
the zeros of $A(t)^2+\lambda^2$; around which points it is $multi-valued.$

We first consider the case where $\lambda>0$ is small. In this case, there is no
turning point on the real axis, and the image  of the real axis by the map
$x\mapsto z(x)$ is the imaginary axis. Let $F(a)$ be the cone-like set
$$
F(a)=\{z\in \C;|\re z|<a|\im z|\}
$$
for $a>0$.
We assume
\begin{description}
\item[(A4)]
For any $\lambda>0$ small, there exist positive constants $\rho(\lambda)$,
$\theta(\lambda)$ and $a(\lambda)$ such that $D(\rho(\lambda), \theta(\lambda))$
contains no turning point and its image by the map $x\mapsto z(x)$ includes
$F(a(\lambda))$.
\end{description}

\begin{theorem}
\label{ref2}
Assume (A1), (A2) and (A4). Then there exists a positive constant $c$ such that
$$
R(\lambda,\e)=\ord \left (e^{-ca(\lambda)/\e}\right ),
$$
as $\e\to 0+$ and $\lambda\to 0+$ with $\frac\e{a(\lambda)}\to 0$.
\end{theorem}

\bigskip

Next we consider the case where $\lambda=i\mu$ with  $\mu>0$ is small. In this case,
there are exactly two turning points $x^*(\mu)$ and $-x^*(\mu)$ on the real axis.
By the map $x\mapsto z(x)$, the real interval $(-x^*(\mu), x^*(\mu))$ is sent to
the imaginary interval $(-z(x^*(\mu)),z(x^*(\mu)))$, and the half line
$(x^*(\mu),\infty)$ (resp. $(-\infty, -x^*(\mu))$ is sent to the half line
$z(x^*(\mu))+\R_+$ (resp. $-z(x^*(\mu))+\R_+$) when the square root in \eq{zx}
is continued from $(-x^*(\mu),x^*(\mu))$ to $(x^*(\mu),\infty)$
(resp. $(-\infty, -x^*(\mu))$ passing through the upper half plane around the turning
point $x^*(\mu)$ (resp. $-x^*(\mu)$). Notice that, as $\mu\to 0$, one has
$x^*(\mu)\to +\infty$ and
$$
|z(x^*(\mu))|\to \int_0^{+\infty}A(x)dx=:z_\infty,
$$
which is a positive finite number. Let $G(b)$ be the complex subdomain of $\C_z$
defined by
$$
\begin{array}{ll}
G(b)=&\{-b<\re z<0, |\im z|<|z(x^*(\mu))|+b)\} \\[8pt]
&\cup\{ \re z\ge 0,|z(x^*(\mu))|<|\im z|<|z(x^*(\mu))|+b\}
\end{array}
$$
for $b>0$.
We assume
\begin{description}
\item[(A5)]
For any $\mu>0$ small, there exist positive constants $\rho(\mu)$,  $\theta(\mu)$
and $b(\mu)$ such that $D(\rho(\mu), \theta(\mu))\cap \{z\in\C;\im z>0\}$ contains
no turning point and its image by the map $x\mapsto z(x)$ includes $G(b(\mu))$.

\end{description}

\begin{theorem}
\label{ev2}
Assume (A1), (A2), (A3) and (A5). Then
there exists a function $m(\mu,\e)$ with asymptotic behavior
$$
m(\mu,\e)=-1+\ord\left (\frac\e{b(\mu)}\right )
$$
as $\e\to 0+$ and $\mu\to 0+$ with  $\frac\e{b(\mu)}\to 0$,
such that $\lambda=i\mu$ is an eigenvalue of $L$ if and only if \eqref{BS} holds.
\end{theorem}

To make the above less opaque, consider the following important subcases.

\begin{example}
\label{pol}
Suppose $A(x)$ satifies (A1), (A2) and
$$
A(x)={C}x^{-d}+r(x) \text{ for } x>1,
$$
with $d>1$, {$C>0$} and $r(x)=o(|x|^{-d-1})$, $r'(x)=o(|x|^{-d-2})$ as
$\re x\to \infty$ in $D_0$. Then, one can take $a(\lambda)=c\lambda$ and
$b(\mu)=c\mu^{1+\frac 1{2d}}$ for some positive constant $c$.
\end{example}

\begin{example}
\label{exp}
Suppose $A(x)$ satifies (A1), (A2) and
$$
A(x)={C}e^{-x^\sigma}+r(x) \text{ for } x>1,
$$
with $\sigma>0$, $C>0$ and $r(x)=o(e^{-x^\sigma})$,
$r'(x)=o(x^{\sigma-1}e^{-x^\sigma})$ as $\re x\to \infty$ in $D_0$.
Then, one can take $a(\lambda)=c\frac{\lambda}{\log\frac{1}{\lambda}}$ and
$b(\mu)=c\mu(\log\frac 1\mu)^{-1+\frac 1{\sigma}}$ for some positive constant $c$.
\end{example}

\begin{corollary}
\label{2.5}
Assume  (A1), (A2) and (A4) with $a(\lambda)\ge c\lambda^{\beta}$ for some $\beta>0$
and $c>0$. Then the reflection coefficient $R(\lambda,\epsilon)$ is exponentially
small with respect to $\epsilon$ uniformly for  $|\lambda|\ge \epsilon^\alpha$
with any $\alpha<1/\beta$. In particular, for potentials  in examples \ref{pol}
and \ref{exp}, $R(\lambda,\epsilon)$ is exponentially small with respect to
$\epsilon$ uniformly for  $|\lambda|\ge \epsilon^\alpha$ with any $\alpha<1$.
\end{corollary}

\begin{corollary}
\label{2.7}
Assume  (A1), (A2), (A3) and (A5) with $b(\mu)\ge c\mu^{\beta}$ for some $\beta>0$
and assume also $|H'(\mu)|\ge c\mu^{\gamma}$ for some $c>0$. Then
\begin{equation}
\label{muk}
|\mu_n(\epsilon)-\mu_n^{\rm WKB}(\epsilon)|=o(\epsilon)
\end{equation}
uniformly for  $|\mu_n|\ge \epsilon^\alpha$ with any $\alpha<1/(\beta+\gamma)$.
In particular,
for potentials of example \ref{pol}, \eq{muk} holds
uniformly for $|\mu|\ge \epsilon^\alpha$ with any $\alpha<\frac d{d+1}$.
For potentials of example \ref{exp},  \eq{muk} holds
uniformly for $|\mu|\ge \epsilon^\alpha$ with any $\alpha<1$.
\end{corollary}

The exact WKB method goes back to the works of Jean Ecalle \cite{e} and Andr\'e
Voros \cite{v}. Here we have made use of the formulation of \cite{gg} and the
further developments in \cite{fln}. Rather than  the usual formal WKB method
which relies on asymptotic series that are in general divergent,  we use a
``resummation" of the series and in fact construct ``exact solutions" in terms
of convergent series,  thus resolving a problem of
``asymptotics beyond all orders". We give a more detailed account of the method
in a later section (for the case with non-zero phase).

\section{The case of non-analytic data}
\label{1-hump}

The previous results employ the exact WKB theory and thus make heavy use of
analyticity. The following results use the alternative method of Olver
\cite{olver1975} which only assumes some smoothness. Rather than resumming,
one proves rigorously that the approximation of the semiclassical Dirac problem
by a parabolic cylinder function is valid in a very precise sense.
 
We have the following results \cite{h+k}. We assume that the function $A$
satisfies the following
\begin{itemize}
\item
$A(x)>0$ for $x\in\mathbb{R}$
\item
$A(-x)=A(x)$ for $x\in\mathbb{R}$
\item
$A$ is in $C^4(\R)$ and of class $C^5$ in a neighborhood of $0$
\item
$xA'(x)<0$ for $x\in\mathbb{R}\setminus\{0\}$
\item
$A''(0)<0$; we set $0<A(0)=:A_{max}$
\item
there exists $\tau>0$ so that as $x\to\pm\infty$
\begin{align*}
A(x)=\mathcal{O}\Big(\tfrac{1}{|x|^{1+\tau}}\Big)
\\
A'(x)=\mathcal{O}\Big(\tfrac{1}{|x|^{2+\tau}}\Big)
\\
A''(x)=\mathcal{O}\Big(\tfrac{1}{|x|^{3+\tau}}\Big)
\end{align*}
\end{itemize}

It has been known for a while \cite{ks} that for such $A$ the discrete spectrum
is imaginary. We have the following.

\begin{theorem}
\label{BoSo}
Suppose that $\lam=i\mu\in\{i\kappa\mid\kappa\in[A_0,A_{max}]\}$ is an
eigenvalue  of the Dirac operator   $\mathfrak{D}_{\e}$ and define
$a>0$ such that $\m=A(a)$. There exists a non-negative integer $n$
(depending both on $\m$ and $\e$) for which the Bohr-Sommerfeld quantization
condition is satisfied, i.e.
\begin{equation}
\label{BS-condition}
\displaystyle\int_{-a}^{a}
\big[A^{2}(x)-\mu^2\big]^{1/2}dx=\pi\Big(n+\tfrac{1}{2}\Big)\e+
\mathcal{O}(\e^{\frac{5}{3}})
\quad\text{as}\quad\e\downarrow0.
\end{equation}
Conversely, for every non-negative integer $n$ such that
$\pi(n+\frac{1}{2})\e \in[0,\tfrac{\pi}{2}\al_0^2]$ ($A_0 =A(\al_0)$)
there exists a unique $\mu$ and consequently an $a>0$ with
$\mu=A(a)$ (where both $\mu$ and $a$ depend on $n,\e$)
satisying
\begin{equation}\nn
\Bigg|
\displaystyle\int_{-a}^{a}
\big[A^{2}(x)-\mu^2\big]^{1/2}dx
-\pi\Big(n+\frac{1}{2}\Big)\e\Bigg|\leq C\e^{\frac{5}{3}}
\end{equation}
with a constant $C$ depending neither on $n$ nor on $\e$;
whence $\lambda=i\mu$ is an eigenvalue
of the Dirac operator  $\mathfrak{D}_{\e}$.
\end{theorem}

The following corollary is a straightforward application of the Theorem \ref{BoSo}
giving the number of eigenvalues of the Dirac operator  in a fixed
(independent of $\e$) interval not containing 0, on the imaginary axis.
\begin{corollary}\label{ev-count-yafa}
Consider an interval $(\m_1,\m_2)\subset[A_{0},A_{max}]$ and take
$a_j$, $j=1,2$ such that $A(a_j)=\m_j$ for $j=1,2$. Then the total number
$\mathcal{N}_{\e}$ of eigenvalues $\lambda=i\mu$ of the Dirac operator
$\mathfrak{D}_{\e}$ lying in the set
$\{i\mu\mid \mu\in(\m_1,\m_2)\}\subset\mathbb{C}$
is equal to
\begin{equation}\label{EV-count}
\mathcal{N}_{\hbar}=\pi^{-1}\big[\fisf(a_1)-\fisf(a_2)\big]
\e^{-1}+R(\e)
\end{equation}
where $|R(\e)|\leq1$ for sufficiently small $\e$.
\end{corollary}

Another straightforward application of Theorem \ref{BoSo} allows us to express
the norming constants of the Dirac operator
\footnote{In particular we see that the asymptotics obtained agree with the known
fact that (because of the symmetry of the potential) the corresponding norming
constant is exactly $(-1)^n$. But of course, our method is easily extended to the
non-symmetric case, see \cite{h+k2}.}. We have the following corollary.
\begin{corollary}
\label{norm-const}
Suppose that $\lam(\e)$ is an eigenvalue of the Dirac operator.
Then there is a non-negative integer $n$ (depending both on $\lam$
and $\e$) such that the corresponding norming constant has
asymptotics
\be\nn
(-1)^n + \asympt(\e^\frac{2}{3})
\quad\text{as}\quad
\e\downarrow0.
\ee
\end{corollary}

Again, it is important for the applications to the semiclassical theory of
the focusing NLS equation, to understand the behavior of the eigenvalues near
0. As done in  \cite{fujii+kamvi} we will present results on a  specific
-but quite inclusive- family of data $A$. We do note however that there is an
explicit integral\footnote{(5.13) in \cite{h+k2}} whose (easy to check)
convergence ensures a good behavior of the eigenvalues near zero for any family
of  data $A$ (defined by explicit prescribed asymptotics at infinity).

\begin{assumption}
\label{near-0-evs-potential-assume}
Suppose further that there are real positive numbers $1<r_+ \leq s_+$,  so that
$$
\frac{C_1^+(x)}{ |x|^{s_+}} \leq A(x) \leq\frac{C_2^+(x)}{ |x|^{r_+}}\quad\text{for}\quad x>0
$$
where $C_1^+, C_2^+$ are bounded functions and $2 r_+ - s_+ > \frac{1}{3}$;
and there are real positive numbers $1<r_- \leq s_-$,  so that
$$
\frac{C_1^-(x)}{ |x|^{s_-}} \leq A(x) \leq\frac{C_2^-(x)}{ |x|^{r_-}}\quad\text{for}\quad x<0
$$
where $C_1^-, C_2^-$ are bounded functions and $2 r_- - s_- > \frac{1}{3}$.
Alternatively,  suppose there are real positive numbers $0<r\leq s$ so that
$$
C_1(x) e^{-|x|^s} \leq A(x) \leq C_2(x) e^{-|x|^r},\quad x\in\R
$$
where $C_1, C_2$ are bounded functions.
\end{assumption}

We have (see \cite{h+k} for the proof)

\begin{theorem}
Let $0<b<\frac{5}{3}$ (independent of $\e$) and consider an intial data function $A$
satisfying the above assumption.
Then for every non-negative integer $n$ such that
$\pi(n+\frac{1}{2})\e$ belongs to
$(0,\tfrac{\pi}{2}\al_0^2(\e))$
there exists a unique $\m_{n}(\e)$ satisfying
\begin{equation}\nn
|\m_n(\e)-\m_n^{WKB}(\e)|=
\asympt\big(\frac{\e^{5/3}}{\log\e}\big),
\quad\text{as}\quad\e\downarrow0
\end{equation}
uniformly for $\mu_n(\e)$ in $[\e^b,A_{max}]$.
\end{theorem}

Furthermore the norming constants are asymptotically $(-1)^n + \asympt(\e^{2/3})$
uniformly as $\e \to 0$. Finally, the reflection coefficient is small enough, even
as we approach fairly close to 0. Indeed.
\begin{theorem}
\label{spectrum+scatter}
Let $A$ satisfy the assumptions at the beginning of this section (only!), take
$0<b<\frac{1}{5}$ (independent of $\e$) and  let $s>0$. Then the reflection
coefficient satisfies
\begin{align}
\label{final-r}
R(\lam(\e),\e) &=
\asympt(\e^{1-sb})
\quad\text{as}\quad
\e\downarrow0
\end{align}
uniformly for $\lambda(\e)$ in any closed interval of $[\e^b,+\infty)$.
\end{theorem}

In \cite{h+k2} we have $extended$ the above results to the case where $A$ is not
symmetric and has many (finite) local maxima. 

\section{The case of non-trivial initial phase}
\label{S-non-zero}

We finally consider the case where $S$ is  not identically zero. We  definitely
need some analyticity or at least meromorphicity properties, since both existing
WKB methods (exact WKB and Olver's method) rely on some kind of analyticity for
the phase. In the examples we have considered $A, S$ admiting a meromorphic
extension in the complex plane.

In \cite{h} and \cite{fujii+hatzi+kamvi} we have studied the special case
$A(x)=S(x)=\sech(2x)$ using a combination of the two WKB methods; our work
was inspired by an older paper by Peter Miller (\cite{mil}) and also some
older numerical observations by Jared Bronski \cite{bron}. Understanding fully
this particular example gives an insight to the most general generic
meromorphic case. Semiclassically the spectrum accumulates along a countable
union of analytic arcs.

\subsection{The general method}
\label{general-method-ev}

This section  presents a modification of the procedure  proposed  in \cite{mil}
for finding the semiclassical eigenvalues of the Dirac (or Zakharov-Shabat)
operator
\begin{equation}
\label{dirac}
\mathfrak{D}_\epsilon=
\begin{bmatrix}
-\frac{\epsilon}{i}\frac{d}{dx} & \omega(x,\epsilon)\\
-\omega^{*}(x,\epsilon) & \frac{\epsilon}{i}\frac{d}{dx}
\end{bmatrix}
\end{equation}
where $\omega$ is the complex potential function
\be\nn
\omega(x,\epsilon)=-iA(x)\exp\Big\{i\frac{S(x)}{\epsilon}\Big\}
\ee
for some real-valued, analytic  functions $A(x)$ and $S(x)$
defined on the real line
[of course $\omega^{*}(x,\epsilon)$ represents the complex conjugate of
$\omega(x,\epsilon)$] such that $A$ and $S'$ are integrable. 
We are interested in the spectral  problem
\begin{equation}
\label{ev-problem}
\mathfrak{D}_\epsilon\textbf{u}=\lambda\textbf{u}
\end{equation}
where $\textbf{u}=[u_1\hspace{3pt}u_2]^T$ is a function from $\R$ to $\C^2$
and $\lambda\in\mathbb{C}$ plays the role of the spectral parameter.
It is clear that (\ref{ev-problem}) can be written equivalently as a first order
system of differential equations, namely
\begin{equation}
\label{miller-ev}
\frac{\epsilon}{i}\frac{d}{dx}\textbf{u}(x,\lambda,\epsilon)=
K(x,\lambda,\epsilon) \textbf{u}(x,\lambda,\epsilon)
\end{equation}
where
\begin{equation}
\label{matrixN}
K(x,\lambda,\epsilon)=
\begin{bmatrix}
-\lambda & \omega(x,\epsilon)\\
\omega^{*}(x,\epsilon) & \lambda \end{bmatrix}.
\end{equation}

Applying first the transformation
\begin{equation*}
\textbf{u}(x,\lambda,\epsilon)=
\begin{bmatrix}
\exp\big\{i\frac{S(x)}{2\epsilon}\big\} & 0\\
0 & \exp\big\{-i\frac{S(x)}{2\epsilon}\big\}
\end{bmatrix}
\tilde{\textbf{v}}(x,\lambda,\epsilon)
\end{equation*}
where $\tilde{\textbf{v}}=[\tilde{v}_1\hspace{3pt}\tilde{v}_2]^T$,  takes the system in
(\ref{miller-ev}) to a new form
\begin{equation}
\label{dirac-reduced}
\frac{\epsilon}{i}\frac{d}{dx}\tilde{\textbf{v}}(x,\lambda,\epsilon)=
\begin{bmatrix}
-\lambda-\frac{1}{2}S'(x) & -iA(x)\\
iA(x) & \lambda+\frac{1}{2}S'(x)
\end{bmatrix}
\tilde{\textbf{v}}(x,\lambda,\epsilon)
\end{equation}
where prime denotes differentiation with respect to $x$.  Next, we apply the mapping
\begin{equation*}
\tilde{\textbf{v}}(x,\lambda,\epsilon)=
\begin{bmatrix}
1 & 1\\
-1 & 1
\end{bmatrix}
\textbf{v}(x,\lambda,\epsilon)
\end{equation*}
where
$\textbf{v}=[v_1\hspace{3pt}v_2]^T$, to finally express the initial system as
\begin{equation}
\label{miller-final}
\frac{\epsilon}{i}\frac{d}{dx}\textbf{v}(x,\lambda,\epsilon)
=M(x,\lambda)\textbf{v}(x,\lambda,\epsilon)
\end{equation}
where
\be
\label{mu-matrix}
M(x,\lambda)=
\begin{bmatrix}
0 & g_+(x,\lambda)\\
-g_-(x,\lambda) & 0
\end{bmatrix}
\ee
\begin{equation}
\label{gigi}
\text{and}\quad g_{\pm}(x,\lambda)=
\mp[\lambda+\frac{1}{2}S'(x)\pm iA(x)].
\end{equation}

The key function to study  is
\begin{align}
\label{turn-pt}
V_0(x,\lambda)
&=
\det M(x,\lambda)\nn\\
&=
g_{-}(x,\lambda)g_{+}(x,\lambda)\nn\\
&=
-[\lambda+\tfrac{1}{2}S'(x)]^2-A^{2}(x).
\end{align}
The zeros of this function play an important role.
\begin{definition}
\label{definition-turn-pt}
For a fixed value of $\lambda\in\mathbb{C}$,  the zeros of $V_0(\cdot,\lambda)$
in $\mathbb{C}$ are called  \textbf{turning points} of (\ref{miller-final}).
\end{definition}

The potential function
$V_0(\cdot,\lambda)$ 
is complex-valued for $\lambda\in\C\setminus\R$ and the turning points are in
general not real.  Assuming that we can extend the functions $A$ and $S$ on the
complex $x$-plane holomorphically, we may consider studying the asymptotics
of solutions of (\ref{miller-ev}) [eventually of (\ref{miller-final})] on some
contour (in the complex $x$-plane), other than the real $x$-axis, connecting
$-\infty$ to $+\infty$ and forcing that  contour to pass through at least one
pair of complex turning points. So, fix a $\lambda\in\C$.  With the above intuition
in mind,  we have the following definition. This is not quite the same as that of
\cite{mil}; it is posed in a form that makes the exact WKB method of the next
subsection directly applicable. But one can  see  that for a case like
$A(x)=S(x)=sech(2x)$ the two are equivalent.
\begin{definition}
\label{x-appropriate-admissible-curve}
Consider a pair $\{x_-(\lambda)$,  $x_+(\lambda)\}$ of complex roots of
$V_0(x,\lambda)=0$.  A contour $C=C^-\cup C^0\cup C^+$ in the complex
$x$-plane, where

\begin{itemize}
\item
the curve $C^-$ connects $-\infty$ to $x_-(\lambda)$
\item
the curve $C^0$ connects $x_-(\lambda)$ to $x_+(\lambda)$ and
\item
the curve $C^+$ connects $x_+(\lambda)$ to $+\infty$,
\end{itemize}
shall be called \textbf{admissible} if the following four conditions hold true:
\begin{itemize}
\item[(i)]
The function $V_0(x,\lambda)$  is holomorphic in $x$, for all $x$ in the
region of the complex $x$-plane enclosed by $C$ and the real $x$-axis
(this is needed to ensure that the approximate eigenfunctions can
be continued back to the real $x$-axis).
\item[(ii)]
For all $x\in C^-$
\begin{equation*}
\Re\bigg\{i
\int_{C^-(x\rightsquigarrow x_-(\lambda))}
\sqrt{-V_0(t,\lambda)}dt
\bigg\}
\end{equation*}
is decreasing as $x$ moves toward $-\infty$,
with the line integral being taken from $x$ to $x_-(\lambda)$ along $C^-$.
\item[(iii)]
For all $x\in C^0$
\begin{equation*}
\Re\bigg\{i
\int_{C^0(x_-(\lambda)\rightsquigarrow x)}
\sqrt{-V_0(t,\lambda)}dt
\bigg\}
=0
\end{equation*}
with the line integral being taken from $x_-(\lambda)$ to $x$ along $C^0$
(this item ensures that the turning points are connected by a path on
which both WKB eigenfunctions are bounded as $\epsilon\downarrow0$).
\item[(iv)]
For all $x\in C^+$
\begin{equation*}
\Re\bigg\{i
\int_{C^+(x_+(\lambda)\rightsquigarrow x)}
\sqrt{-V_0(t,\lambda)}dt
\bigg\}
\end{equation*}
is increasing as $x$ moves toward $+\infty$,
with the line integral being taken from $x_+(\lambda)$ to $x$ along $C^+$.
\end{itemize}
\end{definition}
In all the relations including the square root of $-V_0$ in the definition above, we
choose the branch of $\sqrt{-V_0(x,\lambda)}$ by selecting $C^0$ as the appropriate cut
and imposing that it is  $\sim \lambda$ for  large
$x>0$. $C^0$ is called a Stokes line. Item (iii) in the list above can be reinterpreted
as a differential equation for the path $C^0$ in the complex $x$-plane.
Indeed, if $x=u+iv$, where $u,v\in\mathbb{R}$, then a field of curves is defined by the
differential relation
\begin{equation}
\label{de-real-path}
\Re\bigg\{i\sqrt{-V_0(u+iv,\lambda)}(du+idv)\bigg\}=0.
\end{equation}

Suppose now that for a pair of complex turning points $\{x_-(\lambda),x_+(\lambda)\}$
we have the condition
\begin{equation}
\label{lambda-relation}
\Re\bigg\{i
\int_{C^0(x_-(\lambda)\rightsquigarrow x_+(\lambda))}
\sqrt{-V_0(t,\lambda)}dt
\bigg\}
=0
\end{equation}
and let  $\lambda$ vary. Given a pair of turning points that
are distinct throughout a domain in the complex $\lambda$-plane,
relation (\ref{lambda-relation}) itself determines a curve in the complex
$\lambda$-plane. If $\lambda$ is on one of these curves, the standard
WKB procedure can be expected to apply to determine whether $\lambda$ is
in fact an $o(\epsilon)$ distance away from an eigenvalue (of course, this is
subject to the Stokes line $C^0$ avoiding any singularities
and the existence of appropriate  paths $C^-$, $C^+$).

Therefore, as $\epsilon\downarrow0$, we expect that the discrete eigenvalues
of (\ref{dirac}) will accumulate on the union of curves in the complex
$\lambda$-plane described by formula (\ref{lambda-relation}),
with the union being taken over pairs of complex turning points.
Hence we are led to the following.
\begin{definition}
\label{asympt-spec-arc}
The curves in the complex $\lambda$-plane consisting of $\lambda$-points  that
give rise to  admissible
contours $C$ (on the $x$-plane) will be called \textbf{asymptotic spectral arcs}.
\end{definition}

It follows from this definition that these asymptotic spectral arcs are  unions
of analytic arcs; if there is only a finite number of turning points for each
$\lambda$ these are finite unions of analytic arcs. Applying WKB theory in
\cite{fujii+hatzi+kamvi} to the contour $C$ (admissible contour), we have
rigorously obtained the uniform eigenvalue condition
\begin{equation}
\label{bsqc}
\frac{1}{\pi\epsilon}
\Im\bigg\{i
\int_{C^0(x_-(\lambda)\rightsquigarrow x_+(\lambda))}
\sqrt{-V_0(t,\lambda)}dt
\bigg\}
-\frac{1}{2}
\in\mathbb{Z}
\end{equation}
which can be interpreted as a \textit{Bohr-Sommerfeld quantization rule},
for $\lambda$ on the appropriate asymptotic spectral arc corresponding to the
condition (\ref{lambda-relation}), for all possible admissible choices of
turning points $x_-(\lambda)$ and $x_+(\lambda)$.

\subsection{The exact WKB method}
\label{exact-wkb}
In this and next subsection \ref{exact-wkb-connection-1-tp}, we review the local
theory of the exact WKB method applied to our Zakharov-Shabat operator in a
fixed open set in $\C$.

Let us fix $\lambda$, let $\Omega$ be a simply connected subdomain of the
$x$-complex plane free from turning points, and take a fixed point
$\alpha\in\Omega$. We define a \textit{phase map}
\begin{align}
z(\cdot,\lambda,\alpha):\Omega
&\rightarrow
\mathbb{C}\quad\text{with}\nn\\
\label{phase}
z(x,\lambda,\alpha)=i\int_{\gamma(\alpha\rightsquigarrow x)}
&
\sqrt{-V_0(t,\lambda)}dt
\end{align}
where of course the path of integration is irrelevant and we choose the
branch of $\sqrt{-V_0(x,\lambda)}$ which makes it positive for large $x>0$.
Observe that $z$ is uniquely defined in $\Omega$ and it maps
$\Omega$ conformally to $z(\Omega,\lambda,\alpha)$.

We look for solutions of (\ref{miller-final}) having the form
\begin{equation}
\textbf{v}(x,\lambda,\epsilon,\alpha)=
\exp\bigg\{\pm\frac{z(x,\lambda,\alpha)}{\epsilon}\bigg\}
\tilde{\textbf{w}}^\pm[z(x,\lambda,\alpha)]
\end{equation}
for
$\tilde{\textbf{w}}^\pm=[\tilde{w}^{\pm}_1\hspace{3pt}\tilde{w}^{\pm}_2]^T$.
Substituting this in (\ref{miller-final}), we get the system
\begin{equation}\label{ansatz-1-sys}
\frac{\epsilon}{i}\frac{d}{dz}\tilde{\textbf{w}}^{\pm}(z)
=
\begin{bmatrix}
\pm i & H(z)^{-2}\\
-H(z)^{2} & \pm i
\end{bmatrix}
\tilde{\textbf{w}}^{\pm}(z)
\end{equation}
where [upon recalling (\ref{gigi})] the function $H$ is given by
\begin{equation*}
H(z)=\bigg[\frac{g_-(x,\lambda)}{g_+(x,\lambda)}\bigg]^{1/4}.
\end{equation*}
Finally,  we set
\begin{equation*}
P_{\pm}(z)=
\begin{bmatrix}
H(z)^{-1} & H(z)^{-1} \\
iH(z) & -iH(z)
\end{bmatrix}
\begin{bmatrix}
0 & 1\\
1 & 0
\end{bmatrix}^{\frac{1\pm1}{2}}
\end{equation*}
and apply the transformation
\be\nn
\tilde{\textbf{w}}^{\pm}(z)=P_{\pm}(z)\textbf{w}^{\pm}(z)
\ee
for $\textbf{w}^\pm=[w^{\pm}_{\rm even}\hspace{3pt}w^{\pm}_{\rm odd}]^T$.
Then the system in (\ref{ansatz-1-sys}) is written as
\begin{equation}\label{ansatz-final-sys}
\frac{\epsilon}{i}\frac{d}{dz}\textbf{w}^{\pm}(z)
=
\begin{bmatrix}
0 & \mathcal{H}(z)\\
\mathcal{H}(z) & \mp\frac{2}{\epsilon}
\end{bmatrix}
\textbf{w}^{\pm}(z)
\end{equation}
in which we consider $\mathcal{H}$ to be defined by
\begin{equation}
\mathcal{H}(z)=\frac{\dot{H}(z)}{H(z)}
\end{equation}
and where the dot denotes differentiation with respect to $z$. Easily, one can
arrive at
\begin{equation*}
\mathcal{H}(z)=\frac{d}{dz}\log[H(z)]=
\frac{1}{4}
\frac{A(x)S''(x)-2\lambda A'(x)-A'(x)S'(x)}
{\big\{[\lambda+\tfrac{1}{2}S'(x)]^2+A^{2}(x)\big\}^{3/2}}.
\end{equation*}

In the usual WKB theory,  the vector valued \textit{symbols} $\textbf{w}^{\pm}$
are constructed as a power series in the parameter $\epsilon$,  which is in general
divergent. The exact WKB method consists in the resummation of this divergent series
in the following way. By taking a point $x_0\in\Omega$, setting
\be\nn
z_0\equiv z_0(\lambda)=z(x_0,\lambda,\alpha)
\ee
and postulating the \textit{series ansatz}
\begin{equation}
\label{formal-series}
\textbf{w}^{\pm}=
\sum_{n=0}^{\infty}\textbf{w}^{\pm}_{n}=
\sum_{n=0}^{\infty}
\begin{bmatrix}
w^{\pm}_{2n}\\
w^{\pm}_{2n-1}
\end{bmatrix}\equiv
\begin{bmatrix}
w^{\pm}_{\rm even}\\
w^{\pm}_{\rm odd}
\end{bmatrix}
\end{equation}
we see from (\ref{ansatz-final-sys}) that the functions
$w^{\pm}_{n}, ~-1\leq n\in\mathbb{Z},$ can  be defined inductively by
\begin{equation*}
w_{-1}^\pm\equiv 0,\quad w_0^\pm\equiv 1
\end{equation*}
and for $1\leq n\in\mathbb{Z}$
\begin{equation}\label{recurrence}
\left\{
\begin{array}{rl}
\frac{d}{dz}w_{2n}^\pm(z) &= \mathcal{H}(z)w_{2n-1}^\pm(z) \\
(\frac{d}{dz}\pm\frac{2}{\epsilon})w_{2n-1}^\pm(z) &= \mathcal{H}(z)w_{2n-2}^\pm(z) \\
w_n^\pm(z_0) & =0.
\end{array}
\right.
\end{equation}

These recurrence equations uniquely determine (at least in a neighborhood
of $x_0$) the sequence of scalar functions $\{w_n^\pm(z,\epsilon,z_0)\}_{n=-1}^\infty$
and hence the sequence of vector-valued functions
$\{\textbf{w}_n^\pm(z,\epsilon,z_0)\}_{n=0}^\infty$.
We can write the relations (\ref{recurrence}) in an integral form if we introduce
the following integral operators $\mathcal{I}_\pm$ and $\mathcal{J}$ taking
functions from $\mathscr{H}(z(\Omega,\lambda,\alpha))$ to functions in
$\mathscr{H}(z(\Omega,\lambda,\alpha))$ by
\begin{equation}
\label{I}
\mathcal{I}_\pm[f](z)=
\int_{\Gamma} e^{\pm2(\zeta-z)/\epsilon}\mathcal{H}(\zeta)f(\zeta)d\zeta
\end{equation}
\begin{equation}
\label{J}
\mathcal{J}[f](z)=
\int_{\Gamma}\mathcal{H}(\zeta)f(\zeta)d\zeta
\end{equation}
where $\Gamma\equiv\Gamma(z_0\rightsquigarrow z)$ is the image by $z$
of a path $\gamma\equiv\gamma(x_0\rightsquigarrow x)$ in $\Omega$
connecting $x_0$, $x$; in other words $\Gamma=z(\gamma)$. 
Then (\ref{recurrence}) for $n\geq1$ becomes
\begin{equation}
\label{integral-recurrence}
\left\{
\begin{array}{rl}
w_{2n}^\pm &= \mathcal{J}[w_{2n-1}^\pm] \\
w_{2n-1}^\pm &= \mathcal{I}_{\pm}[w_{2n-2}^\pm].
\end{array}
\right.
\end{equation}

Going back to the initial system (\ref{miller-ev}),
we have constructed (formal) WKB solutions of the form
\begin{multline}
\label{wkb-sols}
\textbf{u}^{\pm}(x,\lambda,\epsilon,\alpha,x_0)=\\
\exp\Big\{\pm\frac{z(x,\lambda,\alpha)}{\epsilon}\Big\}
Q(x,\lambda,\epsilon)
\begin{bmatrix}
0 & 1\\
1 & 0
\end{bmatrix}^{\frac{1\pm1}{2}}
\textbf{w}^{\pm}[z(x,\lambda,\alpha),\epsilon,z_0]
\end{multline}
where $Q(\cdot,\lambda,\epsilon)$ is a matrix-valued function with value
\be\nn
Q(x,\lambda,\epsilon)=
\begin{bmatrix}
\exp\{-iS(x)/(2\epsilon)\} & 0\\
0 & \exp\big\{iS(x)/(2\epsilon)\}
\end{bmatrix}
\begin{bmatrix}
1 & 1\\
-1 & 1
\end{bmatrix}
\begin{bmatrix}
H(z)^{-1} & H(z)^{-1} \\
iH(z) & -iH(z)
\end{bmatrix}.
\ee
Here,  we added the superscript $\pm$ to the solution
$\textbf{u}=[u_1\hspace{3pt}u_2]^T$ -consequently arriving at the
notation $\textbf{u}^{\pm}=[u^{\pm}_1\hspace{3pt}u^{\pm}_2]$-
to distinguish between the two different cases  in the previous
discussion; also we  keep in mind that the WKB solutions
in (\ref{wkb-sols}) depend on a base point $\alpha$ for the phase
map (\ref{phase}) and a base point $x_0$ for the symbols
$\textbf{w}^{\pm}$.

By the definition of a progressive path,
\be\nn
\Omega_\pm=
\{\, x\in\Omega\mid
\exists
\hspace{3pt}(\pm)\text{-progressive path}\hspace{3pt}
\gamma(x_0\rightsquigarrow x)\subset\Omega\,\}.
\ee
We have  the following basic theorem; for the proof,  we refer to \cite{fln}.
\begin{theorem}
\label{formal-to-rigorous}
The exact WKB solutions in (\ref{wkb-sols}) satisfy the following properties.
\begin{enumerate}
\item
The formal series in (\ref{formal-series})  is absolutely convergent in a
neighborhood of $x_0$.\\
\item
For each $N\in\mathbb{N}$, we have
\begin{align*}
\mathbf{w}^{\pm}-\sum_{n=0}^{N-1}\mathbf{w}^{\pm}_{n}
& =
\mathcal{O}(\epsilon^N)
\quad\text{as}\quad\epsilon\downarrow0\\
w^\pm_{\rm even}-\sum_{n=0}^{N-1}w_{2n}^\pm
& =
\mathcal{O}(\epsilon^N)
\quad\text{as}\quad\epsilon\downarrow0\\
w^\pm_{\rm odd}-\sum_{n=0}^{N-1}w_{2n-1}^\pm
& =
\mathcal{O}(\epsilon^N)
\quad\text{as}\quad\epsilon\downarrow0
\end{align*}
uniformly in any compact subset of $\Omega_\pm$.  In particular,
there we have
\begin{align*}
w^\pm_{\rm even}
& =
1+\mathcal{O}(\epsilon)
\quad\text{as}\quad\epsilon\downarrow0\\
\quad w^\pm_{\rm odd}
& =
\mathcal{O}(\epsilon)
\quad\text{as}\quad\epsilon\downarrow0.
\end{align*}
\item
Any two exact WKB solutions with different base points
for the symbol satisfy
\begin{align}
\label{wronsky+-}
{\mathcal W}[{\bf u}^+(x,\lambda,\epsilon,\alpha,x_0),
{\bf u}^-(x,\lambda,\epsilon,\alpha,x_1)]
&=
4i w_{\rm even}^+(z_1,\epsilon,z_0)\\
\label{wronsky++}
{\mathcal W}[{\bf u}^+(x,\lambda,\epsilon,\alpha,x_0),
{\bf u}^+(x,\lambda,\epsilon,\alpha,x_1)]
&=
-4i e^{2z_1/\epsilon}w_{\rm odd}^+(z_1,\epsilon,z_0)
\end{align}
where $x_0$,  $x_1$ are two different points in $\Omega$ and
$z_j=z(x_j,\lambda,\alpha)$ for $j=0,1$.
\end{enumerate}
\end{theorem}

\subsection{Connection around a simple turning point}
\label{exact-wkb-connection-1-tp}

We continue in this section by presenting a local connection formula for exact WKB
solutions near a simple turning point. Let $\alpha$ be a simple turning point, i.e.
\begin{align*}
\det M(\alpha)&=
g_+(\alpha)g_-(\alpha)=0\\
(\det M)'(\alpha)&=
g_+(\alpha)g'_-(\alpha)+g'_+(\alpha)g_-(\alpha)\neq 0
\end{align*}
which means that $\alpha$ is a simple zero of one and only one of either $g_+(x)$ or
$g_-(x)$. Three Stokes lines $l_0$, $l_1$ and $l_2$, numbered in the anti-clockwise
sense, emanate from $\alpha$ and  devide a neighborhood $\omega$ of $\alpha$ into
three regions $\omega_0$ bounded by $l_1$ and $l_2$,  $\omega_1$ bounded by $l_2$
and $l_0$ and $\omega_2$ bounded by $l_0$ and $l_1$. In each of these regions,
we are going to define exact WKB solutions ${\bf u}_0$, ${\bf u}_1$ and ${\bf u}_2$
respectively in such a way that its asymptotic behavior is known near $\alpha$ by
Theorem \ref{formal-to-rigorous}.

We have first to specify the branch of the multi-valued function
$z(x,\lambda,\alpha)$ defined by \eqref{phase}.
Let us take a base point $x_j$  in $\omega_j$ $(j=0,1,2)$.
We put a branch cut on the Stokes line $l_1$, and take the branch of the function
$\sqrt{\det M}$ so that
$\re \,z(x,\lambda,\alpha)$ increases from $x_0$ towards the turning point $\alpha$.
Then consequently $\re \,z(x,\lambda,\alpha)$ decreases from $x_1$ towards  $\alpha$
and increases from $x_2$ towards  $\alpha$.

We define three exact WKB solutions:
\begin{align}
{\bf u}_0(x,\epsilon)={\bf u}^+(x,\lambda,\epsilon,\alpha,x_0), \\
{\bf u}_1(x,\epsilon)={\bf u}^-(x,\lambda,\epsilon,\alpha,x_1), \\
{\bf u}_2(x,\epsilon)={\bf u}^+(x,\lambda,\epsilon,\alpha,x_2).
\end{align}
These solutions are constructed near $x_0$, $x_1$, $x_2$ respectively but can be
extended analytically to $\omega$. They cannot be linearly independent since
the vector space of the solutions is of dimension two. More precisely, we have the
following linear relation.
\begin{proposition}
\label{dep}
Let ${\bf u}_0$, ${\bf u}_1$, ${\bf u}_2$ be three solutions to the differential
equation \eqref{ev-problem} in a region $\omega$. Then the following identity holds:
\be
\label{Widentity}
{\mathcal W}[{\bf u}_1,{\bf u}_2]{\bf u}_0+
{\mathcal W}[{\bf u}_2,{\bf u}_0]{\bf u}_1+
{\mathcal W}[{\bf u}_0,{\bf u}_1]{\bf u}_2=0.
\ee
\end{proposition}
\begin{proof}
If the three determinants ${\mathcal W}[{\bf u}_1,{\bf u}_2]$,
${\mathcal W}[{\bf u}_2,{\bf u}_0]$ and ${\mathcal W}[{\bf u}_0,{\bf u}_1]$ are all 0,
the identity holds obviously. Suppose that at least one of those does not vanish,
say ${\mathcal W}[{\bf u}_1,{\bf u}_2]\ne 0$. Then by taking the determinant of the LHS
of \eqref{Widentity} with ${\bf u}_1$ we find
$$
{\mathcal W}[{\bf u}_1,{\bf u}_2]{\mathcal W}[{\bf u}_0,{\bf u}_1]+
{\mathcal W}[{\bf u}_0,{\bf u}_1]{\mathcal W}[{\bf u}_2,{\bf u}_1]
$$
which is obviously zero. In the same way, we obtain that the determinant of the
LHS of \eqref{Widentity} with ${\bf u}_2$ is zero. Since the pair of solutions
$\{{\bf u}_1, {\bf u}_2\}$ makes a basis of the solution space, this means that
the LHS of \eqref{Widentity} should be zero.
\end{proof}

Let us come back to our exact WKB solutions ${\bf u}_0$, ${\bf u}_1$ and ${\bf u}_2$
defined above and compute the asymptotic behavior of the determinants appearing in
this proposition.

\begin{proposition}
\label{3.567}
As $\epsilon\downarrow0$, the following asymptotics hold true
\begin{align}
\label{01}
&{\mathcal W}[{\bf u}_0,{\bf u}_1]=2i(1+\CO(\epsilon)),
\\
\label{12}
&{\mathcal W}[{\bf u}_1,{\bf u}_2]=-2i(1+\CO (\epsilon)),
\\
\label{20}
&{\mathcal W}[{\bf u}_2,{\bf u}_0]=\mp 2(1+\CO(\epsilon)),
\end{align}
where in (\ref{20}) the upper sign (minus) is to be chosen in the case of $\alpha$ being
a zero of  $g_+$ and the lower sign (plus) in the case that $\alpha$ is a zero of $g_-$. In
particular, any two of the solutions ${\bf u}_0$, ${\bf u}_1$ and ${\bf u}_2$ are linearly
independent for sufficiently small $\epsilon$.
\end{proposition}
\begin{proof}
Using Theorem \ref{formal-to-rigorous} (iii), we find
Using Theorem \ref{formal-to-rigorous} (iii), we find
$$
{\mathcal W}[{\bf u}_1,{\bf u}_2]=-2iw_{\rm even}^+(z_1,\epsilon,z_2)
$$
where $z_j=z(x_j,\lambda,\alpha)$.
Since there exists a $(+)$-progressive path from $x_2$ to $x_1$, one has
$w_{\rm even}^+(z_1,\epsilon,z_2)=1+\CO (\epsilon)$ by Theorem \ref{formal-to-rigorous} (ii),
and hence we obtain \eqref{12}. Similarly we have
$$
{\mathcal W}[{\bf u}_0,{\bf u}_1]=2iw_{\rm even}^+(z_1,\epsilon,z_0)=2i(1+\CO(\epsilon)).
$$

For the computation of the determinant  ${\mathcal W}[{\bf u}_2,{\bf u}_0]$, we have to be
careful since there is a branch cut between $x_2$ and $x_0$. Before applying the
\say{wronskian formula}, we rewrite ${\bf u}_2$ on the Riemann sheet continued from $x_0$
passing across the branch. Let $x$ be a point near $x_2$ and $\hat x$ the same point as $x$
but on that Riemann sheet. Then we have
\begin{align*}
&z(x,\lambda,\alpha)=-z(\hat x,\lambda,\alpha),\\
&H(x)=\mp iH(\hat x),\,\text{ if }\alpha \text{ is a zero of }g_\pm\quad\text{and}\\
&w^\pm (z,\epsilon,z_2)=w^\mp (z,\epsilon,z_2).
\end{align*}
Observe that $\arg (x-\alpha)=\arg(\hat x-\alpha)+2\pi$ and hence the first identity is
obvious since the integrand of $z$ contains a multi-valued function of type
$\sqrt{x-\alpha}$. The second assertion holds due to the fact that $H$ is the product
of $g_+^{-1/4}$ and $g_-^{1/4}$. To verify the third identity, let $w_n^\pm(z)$, $n\ge 1$,
be a family of solutions to the transport equations \eqref{recurrence} satisfying the
initial condition $w_n^\pm (z_2)=0$ at $z_2=z(x_2)$, and set $f_n^\pm(\hat z)=w_n^\pm(z)$,
$n\ge 1$. Inspecting \eqref{recurrence} and using $d/dz=-d/d\hat z$, it is easy to check
that
\begin{align*}
\left (\frac d{d\hat z}\mp\frac 2\epsilon\right )f_{2n+1}^\pm(\hat z)=
\frac{dH(\hat z)/d\hat z}{H(\hat z)}f_{2n}^\pm (\hat z), \\
\frac d{d\hat z}f_{2n+2}^\pm(\hat z)=
\frac{dH(\hat z)/d\hat z}{H(\hat z)}f_{2n+1}^\pm (\hat z)
\end{align*}
with $f_n^\pm (\hat z_2)=0$ at $\hat z_2=z(\hat x_2)$. By uniqueness of solutions, it follows
that $w_n^\pm(z)=w_n^\mp(\hat z)$ with $w_n^\mp(\hat z_2)=0$, which, in view of
\eqref{formal-series}, implies that
$$
{\bf w}^\pm (z,\epsilon,z_2)=
\sum_{n=0}^\infty\left (
\begin{array}{c}
w_{2n}^\pm(z,z_2) \\
w_{2n+1}^\pm(z,z_2)
\end{array}
\right )
=\sum_{n=0}^\infty\left (
\begin{array}{c}
w_{2n}^\mp(\hat z,\hat z_2) \\
w_{2n+1}^\mp(\hat z,\hat z_2)
\end{array}
\right )
={\bf w}^\mp(\hat z,\epsilon,\hat z_2).
$$
Thus, with this point $\hat x_2$ as base point of the symbol, the solution
${\bf u}_2(x,\epsilon)$ is written as
$$
{\bf u}_2(x,\epsilon)=\mp i {\bf u}^-(x,\lambda,\epsilon,\alpha, \hat x_2),\,
\text{ if }\alpha \text{ is a zero of }g_\pm.
$$

We now apply Theorem \ref{formal-to-rigorous} (iii) to compute the desired determinant,
using the above expression. Since there exists a progressive curve from $x_0$ to
$\hat x_2$, we have
$$
{\mathcal W}[{\bf u}_0,{\bf u}_2]=\pm i(\det Q)w_{\rm even}^+(\hat z_2,\epsilon,z_0)=
\pm 2(1+\CO(\epsilon))
$$
when $\alpha$ is a zero of $g_\pm$.
\end{proof}

\subsection{Eigenvalues near a generic point on the asymptotic spectral arc}
\label{bs-at-two-turn-pts}

In this section, we return to the eigenvalue problem \eqref{ev-problem} for the
Zakharov-Shabat operator $\mathfrak{D}_\epsilon$.
We derive the quantization condition for the eigenvalues in a small
neighborhood of a fixed $\lambda_0$, independent of $\epsilon$,
belonging to an asymptotic spectral arc under the following generic condition.

\vspace{0.5cm}

\begin{minipage}{0.1\textwidth}
{\bf (H1):}
\end{minipage}
\hspace{0.01\textwidth}
\begin{minipage}{0.8\textwidth}
$C^0$ contains no other turning point than $x_-(\lambda)$ and $x_+(\lambda)$.
\end{minipage}

\vspace{0.5cm}

Let $C(\lambda_0)$ be the admissible contour
$C(\lambda_0)=C^-(\lambda_0)\cup C^0(\lambda_0)\cup C^+(\lambda_0)$
corresponding to $\lambda_0$ and denote by
$\alpha_0=\alpha(\lambda_0)=x_-(\lambda_0)$,
$\beta_0=\beta(\lambda_0)=x_+(\lambda_0)$
the simple turning points at the extremities of $C^0(\lambda_0)$.  For $\lambda$ in
a small enough neighborhood of $\lambda_0$, the simple turning points
$\alpha=\alpha(\lambda)=x_-(\lambda)$ and $\beta=\beta(\lambda)=x_+(\lambda)$
are still defined and analytic.
From point $\alpha$ there are  three Stokes lines emanating,  namely
$l_0^{\alpha}$, $l_1^{\alpha}$, $l_2^{\alpha}$ (in anti-clockwise order); similarly from
$\beta$ there emerge three Stokes lines which we denote by
$l_0^\beta$, $l_1^\beta$, $l_2^\beta$ (in clockwise order).  For the configuration we refer
to Figure 1.  Suppose $\lambda=\lambda_0$ and that $l_0^{\alpha}$ and
$l_0^{\beta}$ coincide with $C^0(\lambda)$. The five Stokes lines
$C^0(\lambda)$,  $l_1^{\alpha}$, $l_2^{\alpha}$,  $l_1^{\beta}$ and
$l_2^{\beta}$ divide a neighborhood $\omega$ of the admissible contour into
four regions which we denote as follows
\begin{itemize}
\item
$\omega_{\rm left}^{\alpha}$ is bounded by $l_1^{\alpha}$,  $l_2^{\alpha}$
\item
$\omega_{\rm right}^{\beta}$ is bounded by $l_1^{\beta}$,  $l_2^{\beta}$
\item
$\omega_{\rm up}^{\alpha,\beta}$ is bounded by $C^0(\lambda)$,
$l_1^{\alpha}$,  $l_1^{\beta}$ and finally
\item
$\omega_{\rm down}^{\alpha,\beta}$ is bounded by $C^0(\lambda)$,
$l_2^{\alpha}$,  $l_2^{\beta}$.
\end{itemize}
Let us take a point in each of these regions: $x_0^{\alpha}$ in the region
$\omega_{\rm left}^{\alpha}$,  $x_0^{\beta}$ in the region $\omega_{\rm right}^{\beta}$,
$x_1$ in the region $\omega_{\rm down}^{\alpha,\beta}$ and $x_2$ in the
region $\omega_{\rm up}^{\alpha,\beta}$.

\begin{figure}[htbp]
\begin{center}
\includegraphics[width=100mm]{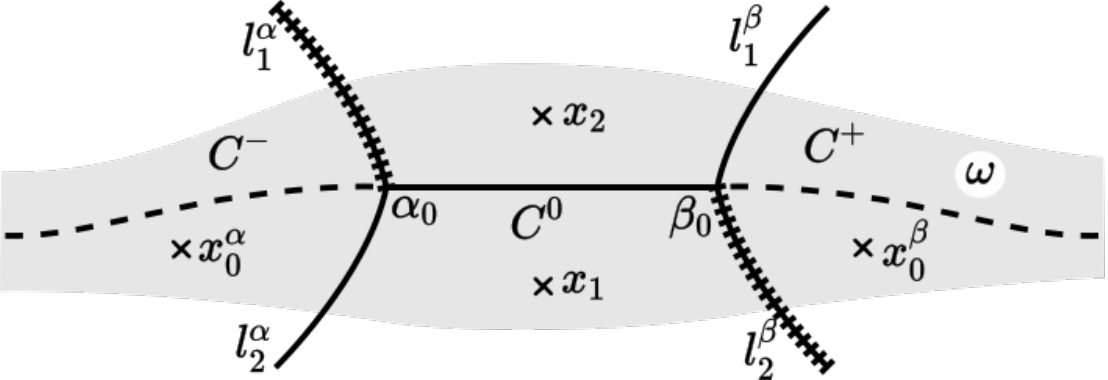}
\caption{Stokes lines near a simple turning point.}
\label{stokes1}
\end{center}
\end{figure}

As in \S \ref{exact-wkb-connection-1-tp}, we have to consider branch cuts; we choose
one on the Stokes line $l_1^{\alpha}$ and one on $l_2^{\beta}$. Assuming that $\re z(x)$
increases towards $\alpha$ in $\omega_{\rm left}^{\alpha}$,  we define exact WKB
solutions as follows.
\begin{align*}
{\bf u}_0^\alpha(x,\epsilon)&={\bf u}^+(x,\lambda,\epsilon,\alpha,x_0^\alpha), \quad
{\bf u}_0^\beta(x,\epsilon)={\bf u}^-(x,\lambda,\epsilon,\beta,x_0^\beta)\\
{\bf u}_1^\alpha(x,\epsilon)&={\bf u}^-(x,\lambda,\epsilon,\alpha,x_1), \quad
{\bf u}_1^\beta(x,\epsilon)={\bf u}^-(x,\lambda,\epsilon,\beta,x_1)\\
{\bf u}_2^\alpha(x,\epsilon)&={\bf u}^+(x,\lambda,\epsilon,\alpha,x_2), \quad
{\bf u}_2^\beta(x,\epsilon)={\bf u}^+(x,\lambda,\epsilon,\beta,x_2)
\end{align*}
Remark that all these exact WKB solutions are chosen so that their asymptotic
expansions in Theorem \ref{formal-to-rigorous} (ii) are valid near the designated
turning points. In particular,  ${\bf u}_0^\alpha(x,\epsilon)$ is a decaying solution
along $C^-(\lambda_0)$ and ${\bf u}_0^\beta(x,\epsilon)$ is a decaying solution
along $C^+(\lambda_0)$. This implies the following lemma.
\begin{lemma}
\label{linear-depend-lemma-of-u}
For any $\lambda_0$ on the asymptotic spectral arcs, there exists a complex
neighbourhood $U=U(\lambda_0)$ of $\lambda_0$ such that a $\lambda\in U$
is an eigenvalue of $\mathfrak{D}_\epsilon$ if and only if
${\bf u}_0^\alpha(x,\epsilon)$ and ${\bf u}_0^\beta(x,\epsilon)$ are linearly
dependent.
\end{lemma}

Applying Proposition \ref{dep} near $\alpha$ and $\beta$, we have
\begin{align}
\label{wronsk1}
{\mathcal W}[{\bf u}_1^\alpha,{\bf u}_2^\alpha]{\bf u}_0^\alpha+
{\mathcal W}[{\bf u}_2^\alpha,{\bf u}_0^\alpha]{\bf u}_1^\alpha+
{\mathcal W}[{\bf u}_0^\alpha,{\bf u}_1^\alpha]{\bf u}_2^\alpha=0, \\
\label{wronsk2}
{\mathcal W}[{\bf u}_1^\beta,{\bf u}_2^\beta]{\bf u}_0^\beta+
{\mathcal W}[{\bf u}_2^\beta,{\bf u}_0^\beta]{\bf u}_1^\beta+
{\mathcal W}[{\bf u}_0^\beta,{\bf u}_1^\beta]{\bf u}_2^\beta=0.
\end{align}
On the other hand, there is an obvious relation between ${\bf u}_j^\alpha$ and
${\bf u}_j^\beta$ for $j=1,2$:
\be
\label{diagrel}
{\bf u}_1^\beta=
e^{z(\beta,\lambda,\alpha)/\epsilon}{\bf u}_1^\alpha,\quad
{\bf u}_2^\beta=
e^{-z(\beta,\lambda,\alpha)/\epsilon}{\bf u}_2^\alpha
\ee
where $z(\beta,\lambda,\alpha)$ is the action integral between $\alpha$ and $\beta$
\be
\label{actionintegral}
z(\beta,\lambda,\alpha)=
i\int_{\gamma(\alpha\rightsquigarrow \beta)}
\sqrt{-V_0(t,\lambda)}dt.
\ee
It follows that ${\bf u}_0^\alpha(x,\epsilon)$ and ${\bf u}_0^\beta(x,\epsilon)$
are linearly dependent if and only if
$$
{\mathcal W}[{\bf u}_2^\alpha,u_0^\alpha]
{\mathcal W}[{\bf u}_0^\beta,{\bf u}_1^\beta]=
{\mathcal W}[{\bf u}_0^\alpha,{\bf u}_1^\alpha]
{\mathcal W}[{\bf u}_2^\beta,{\bf u}_0^\beta]
e^{2z(\beta,\lambda,\alpha)/\epsilon}.
$$
Recall the functions $g_{\pm}(x,\lambda)=\mp[\lambda+\frac{1}{2}S'(x)\pm iA(x)]$.
For a pair of simple turning points $\alpha$, $\beta$, we define an index
$\delta (\alpha,\beta)$ by
\be
\label{dab}
\delta(\alpha,\beta)=
\left\{
\begin{array}{cll}
-1 &\text{if}\quad
g_+(\alpha,\lambda)=g_+(\beta,\lambda)=0 &\text {or}\quad
g_-(\alpha,\lambda)=g_-(\beta,\lambda)=0,\\
1 &\text{if}\quad
g_+(\alpha,\lambda)=g_-(\beta,\lambda)=0 &\text {or}\quad
g_-(\alpha,\lambda)=g_+(\beta,\lambda)=0.
\end{array}
\right.
\ee
Hence we arrive at the following result.

\newpage

\begin{theorem}
\label{BS}
For any $\lambda_0$ on the asymptotic spectral arcs satisfying condition
{\bf (H1)}, there exists a complex neighborhood $U$ of $\lambda_0$ such that
$\lambda\in U$ is an eigenvalue of  $\mathfrak{D}_\epsilon$  if and only if
$$
m(\epsilon)e^{2z(\beta,\lambda,\alpha)/\epsilon}=1,
\quad\text{with}\quad
m(\epsilon)=
\frac{
{\mathcal W}[{\bf u}_0^\alpha,{\bf u}_1^\alpha]{\mathcal W}[{\bf u}_2^\beta,{\bf u}_0^\beta]}
{{\mathcal W}[{\bf u}_2^\alpha,{\bf u}_0^\alpha]{\mathcal W}[{\bf u}_0^\beta,{\bf u}_1^\beta]
}.
$$
As $\epsilon\downarrow0$, $m(\epsilon)$ has the asymptotic behavior
$$
m(\epsilon)=\delta(\alpha_0,\beta_0)+\CO(\epsilon).
$$
\end{theorem}
\begin{proof}
It remains to check the  asymptotic part of the statement.  We already know from
Proposition \ref{3.567} that as $\epsilon\downarrow0$
\be
\label{asympt-wronsk-1}
{\mathcal W}[{\bf u}_2^\alpha,{\bf u}_0^\alpha]=
\mp 2+\CO(\epsilon),\quad
{\mathcal W}[{\bf u}_0^\alpha,{\bf u}_1^\alpha]=
2i+\CO(\epsilon)
\ee
when $\alpha$ is a zero of $g_\pm$. Similarly, as $\epsilon\downarrow0$ we have
\be
\label{asympt-wronsk-2}
{\mathcal W}[{\bf u}_2^\beta,{\bf u}_0^\beta]=2i+\CO(\epsilon),\quad
{\mathcal W}[{\bf u}_0^\beta,{\bf u}_1^\beta]=\mp 2+\CO(\epsilon)
\ee
when $\beta$ is a zero of $g_\pm$.
These asymptotic formulas immediately give the assertion.
\end{proof}

\begin{remark}(The situation at the endpoints $\lambda_D^{(j)},j=1,2$ of the branches)
It is clear from the proof that the statement of Theorem \ref{BS} is still valid for
$\lambda_0$ at a limit point of the asymptotic spectral arcs where $x_+(\lambda_0)$
and $x_-(\lambda_0)$ coalesce to a double turning point.
\end{remark}

For a detailed analysis near the bifurcation point we refer to \cite{fujii+hatzi+kamvi}.

\subsection{The geometry of a particular case}
\label{geometry}

We now focus on 
the very particular  case
\begin{equation}
\label{our-pot}
A(x)=S(x)=\sech(2x),\quad x\in\mathbb{R}.
\end{equation}
There is nothing really special about this example; it just happens that it was the
first case studied numerically  in \cite{bron} and subsequently in \cite{mil}.
\footnote{A different example is studied  in \cite{kcv}.}
But  clearly  what happens here is  true for a very general class
of data $A$ and $S$.

The numerical observation of Bronski was that the asymptotic spectrum consists of five
arcs. There is a vertical line segment $[-i\mu_s, i \mu_s]$ where $\mu_s$ is found to
be approximately $0.28$. There are also four slightly curved arcs, lying in  different
quadrants of the spectral plane. They each connect $i\mu_s$ to (a value of) $\lam_{D} $,
where
\be\label{lamda-double}
\lam_D=
i\sigma\sqrt{\frac{1}{2}+\frac{1+i\tau\sqrt{7}}{8}}
\Bigg(1-\frac{1+i\tau\sqrt{7}}{4}\Bigg)
\quad\text{where}\quad\sigma,\tau=\pm1
\ee
are  the solutions of the transcendental equations
\be\nn
\tanh(2x)=\frac{1+i\tau\sqrt{7}}{4i\sigma}
\quad\text{where}\quad\sigma,\tau=\pm1
\ee
We adopt the notation
$\lam_D^{(j)}$, $j=1,\dots,4$, by requiring that  the point $\lam_D^{(j)}$ is located in
the $j^{\rm th}$ quadrant in the complex $\lam$-plane.

The numerics in \cite{mil} will not be presented here in detail. Summarizing, one
considers different values of $\lambda$ in the complex plane and then for each such
value one investigates whether there exists an admissible path in the complex
x-plane joining $-\infty$ and $+\infty$.

Our main contribution in \cite{fujii+hatzi+kamvi} was to provide rigorous asymptotics
for the eigenvalues lying on the asymptotic spectral arcs. Our results are very much
like the Bohr-Sommerfeld relations in the zero phase case!

We conclude this subsection  with the  quantization condition of eigenvalues along
the union of the asymptotic spectral arcs away from  a neighborhood of the bifurcation
point $ \lambda_\otimes = i \mu_s$. For the analogous statement near the bifurcation
point (and the proofs) we refer to Theorem 2.23 in \cite{fujii+hatzi+kamvi}.

\begin{theorem}
For any $\epsilon$-independent $\delta>0$ and
$\lambda_0\neq0$ on the union of the three asymptotic spectral arcs,
 satisfying
$|\lambda_0-\lambda_\otimes|>\delta$, there exist a complex neighborhood $U$ of $\lambda_0$
and a function $m_{j}(\epsilon)$ with a uniform asymptotic behavior
$m_{j}(\epsilon)=-1+\CO(\epsilon)$ in $U$, such that $\lambda\in U$ is an eigenvalue of
$\mathfrak{D}_\epsilon$  if and only if
$$
m_{j}(\epsilon)e^{2z_j(\beta,\lambda,\alpha)/\epsilon}=1.
$$
The index $j=1,2,3$ denotes the choice of the asymptotic spectral arc. The phase $z_j$
also depends on that choice (see (4.13)).
\end{theorem}

\subsection{Eigenvalues near zero}
\label{evs-near-zero}
As in the previous sections the distribution of the eigenvalues near zero needs special
care. In \cite{fujii+hatzi+kamvi} we have included a discussion of this matter. The exact
WKB method does not (seem to) provide the estimates necessary for the asymptotic analysis
of the inverse scattering problem. We have instead made use of the Olver theory, as
further developed in \cite{olver1978}. In the special $sech$ case there is no
complication, because of the explicit exponential behavior of the initial data at
infinity.

\end{document}